\documentclass[12pt,oneside]{amsart}
\usepackage[margin=3cm]{geometry}
\usepackage[utf8]{inputenc}
\usepackage{amsmath}
\usepackage{tikz-cd}
\usepackage{bm}
\usepackage{amssymb}
\usepackage{mathtools}
\usepackage{amsthm}
\usepackage{graphicx}
\usepackage{comment}
\usepackage{bbold}
\usepackage{comment}
\usepackage{appendix}
\usepackage{caption}
\usepackage{subcaption}
\usepackage{hyperref}
\usepackage{upgreek}
\usepackage{enumitem}
\usepackage{comment}
\usepackage{tikz}
\usepackage{pict2e,mathrsfs}

\title{Additive invariants of Finite-dimensional DG-algebras}
\author{Isambard Goodbody}
\address{Isambard Goodbody, Zeeman Building, University of Warwick, Coventry, CV4 7AL, igoodbody.maths@gmail.com}
\date{\today}

\DeclareUnicodeCharacter{0327}{\c{}}

\DeclareMathOperator{\End}{End}

\DeclareMathOperator{\add}{add}

\DeclareMathOperator{\perf}{perf}

\DeclareMathOperator{\fd}{fd}

\begin{document}

 \setcounter{MaxMatrixCols}{20}

\usetikzlibrary{matrix}
\usetikzlibrary{shapes}

\theoremstyle{plain}
\newtheorem{prop}{Proposition}[section]
\newtheorem{lemma}[prop]{Lemma}
\newtheorem{theorem}[prop]{Theorem}
\newtheorem{cor}[prop]{Corollary}
\newtheorem{conj}[prop]{Conjecture}
\newtheorem{question}[prop]{Question}
\newtheorem{prob}[prop]{Problem}

\theoremstyle{definition}
\newtheorem{defn}[prop]{Definition}
\newtheorem{ass}[prop]{Assumption}
\newtheorem{cons}[prop]{Construction}
\newtheorem{ex}[prop]{Example}
\newtheorem{remark}[prop]{Remark}
\newtheorem{setup}[prop]{Set up}
\newtheorem*{ack}{Acknowledgements}
\maketitle

\tikzset{
    vert/.style={anchor=south, rotate=90, inner sep=.5mm}
}

\newcommand{\rightarrowdbl}{\rightarrow\mathrel{\mkern-14mu}\rightarrow}

\newcommand{\xrightarrowdbl}[2][]{%
  \xrightarrow[#1]{#2}\mathrel{\mkern-14mu}\rightarrow
}

\begin{abstract}
Orlov has shown that additive invariants such as K-theory and Hochschild homology of smooth finite-dimensional DG-algebras behave similarly to those of finite-dimensional algebras. The similarities only go so far: there are examples of finite-dimensional DG-algebras with vanishing and uncountable Grothendieck groups. We show that for DG-quivers with relations, the Grothendieck groups are free of the expected rank and that, in the smooth case, all additive invariants are determined by the base field. We use our results to rule out certain finite-dimensional DG-algebras from being phantoms.
\end{abstract}
\tableofcontents

\section{Introduction}

In \cite{Orl20}, the Grothendieck group of a smooth finite-dimensional DG-algebra was shown to be free and of finite but unknown rank. The goal of this note was to address the following problem.

\begin{prob}[{\cite[Problem 2.22]{Orl20}}] \label{Orlovsprob}
Find the rank of the Grothendieck group of smooth finite-dimensional DG-algebras.
\end{prob}

The Grothendieck group of a finite-dimensional algebra (smooth or not) is free of rank given by the number of indecomposable projectives or equivalently the number of simples. The Grothendieck group of a finite-dimensional DG-algebra can behave very differently: there are smooth finite-dimensional DG-algebras with vanishing Grothendieck groups (Section \ref{phantoms}) and there are finite-dimensional DG-algebras whose Grothendieck groups are uncountable (Remark \ref{badremarkK0}). To prove some positive results we focus on the following class of finite-dimensional DG-algebras.

\begin{defn}
We say that a DG-algebra $A$ over a field $k$ is a DG-quiver with relations if the following conditions hold. 
\begin{enumerate}
\item The underlying algebra of $A$ is finite-dimensional and given by $kQ/I$ where $Q$ is a finite quiver and $I$ is an admissible ideal.

\item The differential $d$ of $A$ satisfies $d(e) = 0$ for every idempotent $e$. 
\end{enumerate}
\end{defn}

We let $N_{proj}$ denote the number of isomorphism classes of indecomposable summands of $A$ in its derived category. The maximal semisimple quotient $A/J_+$ of $A$ was constructed in \cite{Orl20} and we let $N_{sim}$ denote the number of isomorphism classes of indecomposable summands of $A/J_+$.  Recall that an additive invariant is any invariant which sends semi-orthogonal decompositions to direct sums. Examples include $K$-theory, Hochschild homology and its variants.
\begin{theorem}[Corollary \ref{quivercor}]
Let $A$ be a DG-quiver with relations. There are isomorphisms
\[
\mathbb{Z}^{N_{proj}} \simeq K_0(A) \xrightarrow{\sim} K_0(A/J_+) \simeq \mathbb{Z}^{N_{sim}}
\]
Suppose $A$ is additionally smooth, then
\[
U(A) \simeq U(A/J_+) \simeq U(k)^{\bigoplus N_{sim}}
\]
for any additive invariant $U$. 
\end{theorem} 

A phantom is a smooth and proper DG-algebra $A$ such that all of the additive invariants of $A$ vanish. A result of \cite{Efi23} shows that every phantom is quasi-isomorphic to a finite-dimensional DG-algebra. Combined with some results on Hochschild homology, we can rule out certain finite-dimensional DG-algebras from being phantoms.

\begin{theorem} \label{introphantom} Suppose $A$ is a phantom finite-dimensional DG-algebra.
\begin{enumerate}
\item For every primitive idempotent $e$ of $A$ $d(e) \neq 0$.
\item The map $H^0(A) \to H^0(A/J_+)$ is not epi.
\item The underlying algebra of $A$ is not basic.
\item $HH_1(H^\ast(A)) \neq 0$. 
\item If $k$ is perfect, then $H^\ast(A)$ has infinite global dimension.
\end{enumerate}
\end{theorem}

\begin{ack}
Some of this work was completed during my PhD and I'm grateful to my supervisor at the time, Greg Stevenson, for his guidance. 
\end{ack}

\section{Preliminaries}

Throughout $k$ will denote a field. We let $A^e \coloneq A \otimes_k A^{op}$ denote the enveloping algebra of $A$. The derived category of $A$ will be denoted $\mathcal{D}(A)$, the perfect complexes $\mathcal{D}^{\perf}(A)$ and the modules with finite-dimensional cohomology $\mathcal{D}_{\fd}(A)$. The following notions are noncommutative analogues of smooth and proper schemes. See \cite{Orl14} for results making this precise. 

\begin{defn} Let $A$ be a DG-algebra.
\begin{enumerate}
\item $A$ is smooth if $A \in \mathcal{D}^{\perf}(A^e)$.
\item $A$ is proper if $H^\ast(A)$ is finite-dimensional.
\item $A$ is finite-dimensional if $A^i$ is finite-dimensional for every $i$ and $A^i = 0$ for $\lvert i \rvert >> 0$.
\item A proper DG-algebra $A$ is regular if $\mathcal{D}_{\fd}(A) = \mathcal{D}^{\perf}(A)$.
\end{enumerate}
\end{defn}

\begin{remark} Every finite-dimensional DG-algebra is clearly proper but the converse is not true (e.g. \cite[Remark 2.4]{goo23}). Recall that a DG-algebra $A$ is connective if $H^i(A) = 0$ for $i > 0$ and coconnective if $H^i(A) = 0$ for $i < 0$. By \cite[Appendix A]{Goo26} or \cite[Corollary 3.12]{RS20}  any proper connective DG-algebra is quasi-isomorphic to a finite-dimensional DG-algebra. It follows from \cite{NK11}, that any regular proper coconnective DG-algebra with semi-simple $H^0$ is quasi-isomorphic to a finite-dimensional DG-algebra. 
\end{remark}

\begin{remark} \label{geoeg}
By \cite[Corollary 2.20]{Orl20}, finite-dimensional DG-algebras are exactly those proper DG-algebras which admit categorical resolutions with full exceptional collections. In particular, a projective scheme $X$ is derived equivalent to a finite-dimensional DG-algebra if and only if it admits such a categorical resolution. This is conjecturally related to rationality of $X$ (see \cite[Conjecture 3.6]{Orl20}, \cite{LZ26}, \cite{Wei16}). In particular, if $X$ is a rational curve, it follows from \cite{BDG17} or \cite{KL15} (see also \cite[Proposition 4.1]{Wei16}) that $\mathcal{D}^{\perf}(C)$ is equivalent the perfect derived category of a finite-dimensional DG-algebra. 
\end{remark}

\begin{defn}[{\cite[Definition 2.3]{Orl20}}]
If $A$ is a finite-dimensional DG-algebra, let $J$ denote the radical of the underlying algebra of $A$ and
\[
J_- \coloneq \{ x \in J \mid d(x) \in J \} \quad J_+ = J + d(J)
\]
\end{defn}

\begin{remark}
The DG-ideals $J_-$ and $J_+$ are quasi-isomorphic. The ideal $J_-$ is nilpotent and $A/J_+$ is derived equivalent to a product of division algebras by \cite[Proposition 2.16]{Orl20}. In this way $A/J_+ \simeq A/J_-$ behaves like the maximal semisimple quotient of a finite-dimensional algebra. 
\end{remark}

The above construction can be used to prove the following. We don't know if the statement is true for all proper DG-algebras. 

\begin{prop}
Let $A$ be a finite-dimensional DG-algebra such that $A/J_+$ is $k$-separable. Then the following are equivalent.
\begin{enumerate}
\item $A$ is regular
\item $A$ is smooth 
\item $\mathcal{D}^{\perf}(A)$ admits a strong generator.
\end{enumerate}
\end{prop}

\begin{proof}
If $A$ is regular, then $A/J_+$ is a strong generator of $\mathcal{D}^{\perf}(A)$. Indeed this follows from a radical filtration argument (e.g.\ \cite[Theorem 2.14]{goo23}) combined with \cite[Proposition 2.5]{Orl20}. If $\mathcal{D}^{\perf}(A)$ has a strong generator, then the inclusion $\mathcal{D}^{\perf}(A) \hookrightarrow \mathcal{D}^{\perf}(\tilde{A})$ where $\tilde{A}$ is the Auslander algebra of \cite{Orl20}, admits a right adjoint by \cite[Theorem 1.3]{VB02}. As $A/J_+$ is $k$-separable, $\tilde{A}$ is smooth by \cite[Theorem 2.19]{Orl20}. Hence $A$ is smooth by \cite[Theorem 3.24]{LS14}. If $A$ is smooth then $A$ is regular by \cite[Lemma 3.8]{KS22}. 
\end{proof}

Being finite-dimensional places strong constraints on invariants. The Grothendieck group of a DG-algebra $A$ is defined as $K_0(A) \coloneq K_0(\mathcal{D}^{\perf}(A))$. The Hochschild homology of $A$ is $HH_{\ast}(A)\coloneq H^{-\ast}(A \otimes^{\mathbb{L}}_{A^e} A)$. Both of the following results fail for general smooth and proper DG-algebras. 

\begin{theorem} \label{niceinv} Let $A$ be a regular finite-dimensional DG-algebra.
\begin{enumerate}
\item (\cite[Corollary 2.21]{Orl20}) The Grothendieck group $K_0(A)$ is free and finitely generated.

\item (\cite[Corollary 2.23]{goo23}) If $A/J_+$ is $k$-separable, $HH_n(A) = 0$ for $n \neq 0$.
\end{enumerate}
\end{theorem}

In order to address Problem \ref{Orlovsprob} we make the following definitions.

\begin{defn}
Let $A$ be a finite-dimensional DG-algebra. 
\begin{enumerate}
\item Let $N_{proj}$ be the number of isomorphism classes of indecomposable summands of $A$ in $\mathcal{D}(A)$.

\item Let $N_{sim}$ be the number of isomorphism classes of indecomposable summands of $A/J_+$ in $\mathcal{D}(A)$. 
\end{enumerate}
\end{defn}

\section{An Upper Bound}

In this section, we provide some conditions for the Grothendieck group of a finite-dimensional DG-algebra to be finitely generated. 

\begin{remark} \label{badremarkK0}
We can produce examples of finite-dimensional DG-algebras with badly behaved invariants using Remark \ref{geoeg}. If $X$ is the nodal cubic over $\mathbb{C}$, there is a finite-dimensional DG-algebra such that $\mathcal{D}^{\perf}(X) \simeq \mathcal{D}^{\perf}(A)$. Hence $K_0(A) \simeq K_0(X) = \mathbb{Z} \oplus \mathbb{Z} \oplus \mathbb{C}^{\times}$ which is not finitely generated.
\end{remark}

\begin{lemma} \label{upperlemma}
Let $A$ be a finite-dimensional DG-algebra such that $d(e) = 0$ for each idempotent $e \in A$. Then any $M \in \mathcal{D}^{\perf}(A)$ is quasi-isomorphic to a module with a filtration. 
\[
0 = M_0 \subseteq M_1 \subseteq \dots \subseteq M_n = M
\]
such that $M_i/M_{i-1}$ is a finite sum of shifts of summands of $A$. 
\end{lemma}

\begin{proof}

Recall that $M$ is a summand in $\mathcal{D}(A)$ of a finite semi-free $A$-module $F$ i.e.\ one with a filtration as in the statement. As in the proof of \cite[Proposition 2.5]{Orl20}, we can assume that $M$ is a strict summand of $F$ as a DG-module not just up to homotopy. 

Since $F$ is semifree, there is a finite set $S = \sqcup_{0 \leq i \leq n} S_i \subseteq F$ of homogenous elements such that the underlying graded modules satisfy $F^{\#} = \bigoplus_{s \in S} sA^{\#}$ and such that for each $i$, $d(S_i) \subseteq F_{i-1} \coloneq \bigoplus_{s \in S_{0}\sqcup \dots \sqcup S_{i-1}} sA^{\#}$. Since $M^{\#}$ is a summand of $F^{\#}$ and $A^{\#}$ is a finite-dimensional algebra, there exist a subset $S' \subseteq S$ and non-zero idempotents $e_s$ for each $s \in S'$ such that 
\[
M^{\#} = \bigoplus_{s \in S'} se_sA^{\#}
\]
Let $S'_i \coloneq S' \cap S_i$ and consider the submodules $M_i \coloneq F_i \cap M$ of $M$. We have that $M_i/M_{i-1}^{\#} \simeq \bigoplus_{s \in S_i'} se_s A^{\#}$ on which the differential is given by $d(se_sa) = sd(e_sa) = se_s d(a)$. Therefore $M_i/M_{i-1} \simeq \bigoplus e_s A$ and each each $e_sA$ is a summand of $A$ as a DG-module. 
\end{proof}

\begin{remark}
Recall that a subcategory of a triangulated category is thick if it is closed under extensions, shifts and direct summands. If $A$ is any DG-algebra, then $A$ generates $\mathcal{D}^{\perf}(A)$ as a thick subcategory. If $A$ is a finite-dimensional algebra, then the indecomposable summands of $A$ generate $\mathcal{D}^{\perf}(A)$ as a triangulated category i.e.\ it is not necessary to take retracts. The following result generalises this fact.
\end{remark}

\begin{theorem}\label{upperboundthm}
Let $A$ be a finite-dimensional DG-algebra such that  $d(e) = 0$ for every idempotent $e$. The perfect derived category $\mathcal{D}^{\perf}(A)$ is generated by the indecomposable summands of $A$ as a triangulated category and there is a surjection
\[
\mathbb{Z}^{N_{proj}} \twoheadrightarrow K_0(A).
\]

\end{theorem}

\begin{proof}
The existence of the filtration in Lemma \ref{upperlemma} implies every object of $\mathcal{D}^{\perf}(A)$ can be built from the indecomposable summands of $A$ using extensions and shifts (and without taking retracts). It follows that the class of any object in $K_0(A)$ is a linear combination of the classes of the indecomposable summands of $A$.
\end{proof}

\begin{remark} \label{smoothupper}
If $A$ is any regular finite-dimensional DG-algebra, then a radical filtration argument (e.g.\ \cite[Theorem 2.14]{goo23}) implies that the indecomposable summands of $A/J_+$ generate $\mathcal{D}^{\perf}(A)$ as a triangulated category. Hence $K_0(A) \simeq \mathbb{Z}^N$ where $N \leq N_{sim}$. 
\end{remark}

\section{A Lower Bound}

In this section we produce non-zero classes in the Grothendieck group of a finite-dimensional DG-algebra $A$ by lifting them from $K_0(A/J_+)$. There are finite-dimensional DG-algebras with vanishing Grothendieck group (see Section \ref{phantoms}) and so additional assumptions on $A$ are necessary.

\begin{remark} 
Let $A$ be a finite-dimensional DG-algebra, then by \cite[Proposition 2.16]{Orl20}, there are division algebras $D_i$ such that $\mathcal{D}(A/J_+) \simeq \mathcal{D}(D_1) \times \dots \times \mathcal{D}(D_n)$. The equivalence sends $D_i$ to an indecomposable summand $S_i$ of $A/J_+$. Hence 
\[
K_0(A/J_+) = \bigoplus \mathbb{Z}[S_i] \simeq \mathbb{Z}^{N_{sim}}
\]
\end{remark}

\begin{defn} Let $A$ be a finite-dimensional DG-algebra.
\begin{enumerate}
\item We say that an idempotent $e \in H^0(A/J_+)$ is liftable if there is an idempotent $\tilde{e}$ of $H^0(A)$ which is sent to $e$ under the map $H^0(A) \to H^0(A/J_+)$. 

\item We let $N_{lift} \leq N_{sim}$ denote the number of isomorphism classes of indecomposable summands of $A/J_+$ such that the corresponding idempotent of $H^0(A/J_+)$ lifts to an idempotent of $H^0(A)$. 
\end{enumerate}
\end{defn}

\begin{remark} \label{idempotentremark} \ \begin{enumerate}
\item Any idempotent lifts along quotients of finite-dimensional algebras and so $e$ being liftable is equivalent to $e$ being in the image of the map $H^0(A) \to H^0(A/J_+)$. 

\item If $e \in H^0(A/J_+)$, then it lifts to an idempotent $e\in A/J_+$ which lifts to an idempotent $\tilde{e} \in A$. If $d(\tilde{e}) = 0$ then $[\tilde{e}] \in H^0(A)$ lifts $e$. So if $d(e) = 0$ for every (primitive) idempotent of $A$, then all idempotents lift. 
\end{enumerate}
\end{remark}

\begin{theorem} \label{lowerbound}
Let $A$ be a finite-dimensional DG-algebra and suppose $S$ is an indecomposable summand of $A/J_+$ whose corresponding idempotent $e \in H^0(A/J_+)$ is liftable. Then $[S]$ is in the image of the map $K_0(A) \to K_0(A/J_+)$. Hence there is a surjective map. 
\[
K_0(A) \twoheadrightarrow \mathbb{Z}^{N_{lift}} \subseteq K_0(A/J_+)
\]
\end{theorem}

\begin{proof}
The lift $\tilde{e} \in H^0(A)$ of $e$ splits as a morphism $\tilde{e}: A \to A$ in $\mathcal{D}(A)$ via some summand $P$ of $A$. The functor $-\otimes_A^{\mathbb{L}} A/J_+$ sends $\tilde{e}$ to $e$ and so $P \otimes_A^{\mathbb{L}} A/J_+$ and $S$ are both splittings of the idempotent morphism $e$. Therefore they are isomorphic in $\mathcal{D}(A/J_+)$. Hence $[P] \in K_0(A)$ is sent to $[S] \in K_0(A/J_+)$. 
\end{proof}
Combining Theorems \ref{lowerbound} and \ref{upperboundthm} we deduce the following.

\begin{cor} \label{squeeze}
Let $A$ be a finite-dimensional DG-algebra such that $d(e) = 0$ for every idempotent $e$. Then there are surjections
\[
\mathbb{Z}^{N_{proj}} \twoheadrightarrow K_0(A) \twoheadrightarrow \mathbb{Z}^{N_{sim}}
\]
\end{cor}

Combining Theorem \ref{lowerbound} and Remark \ref{smoothupper} we deduce the following.

\begin{cor} \label{smoothK0iso}
Let $A$ be a regular finite-dimensional DG-algebra then $K_0(A) \simeq \mathbb{Z}^N$ where $N_{lift} \leq N \leq N_{sim}$. If either
\begin{enumerate}
\item  $d(e) = 0$ for every idempotent $e$ of $A$
\item or $H^0(A) \twoheadrightarrow H^0(A/J_+)$
\end{enumerate}
then the map $K_0(A) \xrightarrow{\sim} K_0(A/J_+)\simeq \mathbb{Z}^{N_{sim}}$ is an isomorphism.
\end{cor}

\section{The Simple-Projective bijection}

For a finite-dimensional algebra, there is a bijection between the isomorphism classes of indecomposable projectives and those of the simples. Another way of stating this simple-projective bijection is that the map $K_0(A) \to K_0(A/J)$ is an isomorphism.

\begin{theorem} [Simple-Projective Bijection] \label{spbij}
Suppose $A$ is a finite-dimensional DG-algebra such that $H^0(A) \to H^0(A/J_+)$ is surjective, then $N_{proj} = N_{sim} = N_{lift}$. 
\end{theorem}

\begin{proof}
By Theorem \ref{lowerbound} and Remark \ref{idempotentremark}, the functor $-\otimes_A^{\mathbb{L}} A/J_+$ induces an epimorphism from the set of isomorphism classes of summands of $A$ to isomorphism classes of summands of $A/J_+$. Let $S:= P\otimes_A^{\mathbb{L}} A/J_+$ and $S' := P' \otimes_A^{\mathbb{L}} A/J_+$ be summands of $A/J_+$ and $g: S \to S'$ an isomorphism in $D(A/J_+)$. Then the map 
\[
A/J_+ \twoheadrightarrow S \xrightarrow{g} S' \hookrightarrow A/J_+
\]
is an element $h \in H^0(A/J_+)$ and so lifts to a map $f: A \to A$. Then $-\otimes_A^{\mathbb{L}} A/J_+$ sends
\[
(P \hookrightarrow A \xrightarrow{f} A \twoheadrightarrow P') \mapsto (S \hookrightarrow A/J_+ \twoheadrightarrow S \xrightarrow{g} S' \hookrightarrow A/J_+ \twoheadrightarrow S') = (S \xrightarrow{g} S')
\]
By \cite[Theorem 3.3]{goo23}, the functor $-\otimes_A^{\mathbb{L}} A/J_+$ is conservative and so $P \simeq P'$. Hence the above map on isomorphism classes is bijection. If $P$ is a decomposable summand of $A$ then the associated summand of $A/J_+$ clearly is too. Suppose $S$ is a summand of $A/J_+$ with lift $P$ and $S \simeq S_1 \oplus S_2$. Then we can lift $S_1$ and $S_2$ to $P_1$ and $P_2$. Then we have that $P$ and $P_1 \oplus P_2$ are both lifts of $S$ and so $P \simeq P_1 \oplus P_2$ by the above uniqueness. Hence $P$ is indecomposable if and only if $P \otimes_A^{\mathbb{L}} A/J_+$ is. Therefore $N_{proj} = N_{sim}$. By Remark \ref{idempotentremark}, $N_{sim}  = N_{lift}$.  
\end{proof}

From Theorem \ref{spbij} and Corollary \ref{squeeze} we deduce the following.

\begin{cor} \label{K0isocor}
Suppose $A$ is a finite-dimensional DG-algebra such that $H^0(A) \to H^0(A/J_+)$ is epi and $d(e) = 0$ for every idempotent $e \in A$. There are isomorphisms
\[
 \mathbb{Z}^{N_{proj}} \simeq K_0(A) \xrightarrow{\sim} K_0(A/J_+) \simeq \mathbb{Z}^{N_{sim}}
\]
\end{cor}

Using the argument of \cite[Theorem 4.6]{RS20} we extend Corollary \ref{K0isocor} to all additive invariants in the smooth case. 

\begin{lemma} \label{nicetrick}
Let $f: A \to B$ be a map of smooth and proper DG-algebras such that the induced map 
\[
K_0(A^{op} \otimes_k A) \xrightarrow{\sim} K_0(B^{op} \otimes_k B)
\]
is an isomorphism. Then $f$ induces an isomorphism $U(A) \simeq U(B)$ for every additive invariant $U$.
\end{lemma}

\begin{proof}
It is enough to produce an isomorphism on additive motives as in \cite{Tab15}. As $A$ and $B$ are smooth and proper the morphisms from $U_{add}(A)$ to $U_{add}(B)$ in the category of additive motives is given by $K_0(A^{op} \otimes B)$. Factoring the isomorphism in two different ways implies that postcomposing by $U_{\add}(f)$ in the category of motives induces an epi and a mono
\[
K_0(B^{op} \otimes A) \twoheadrightarrow K_0(B^{op} \otimes B)
\quad K_0(A^{op} \otimes A) \hookrightarrow K_0(A^{op} \otimes B)
\]
As the first map is epi, $U_{add}(f)$ admits a right inverse $g$. As the second map is mono, $g$ is also a left inverse. 
\end{proof}

\begin{lemma} \label{J=J_+}
Let $A$ be a finite-dimensional DG-algebra such that $d(e) = 0$ for each idempotent $e$ of $A$. Then $A$ is quasi-isomorphic to a DG-algebra $A'$ such that $d(f) = 0$ for each idempotent $f$ of $A'$ and $J(A') = J(A')_+$  and $A/J_+ \cong A'/J$. 
\end{lemma}

\begin{proof}
As $A/J$ is a product of matrix algebras, the map $A/J \twoheadrightarrow A/J_+$ is a split surjection and so if $J \neq J_+$ there must be an idempotent $e \in J_+$. Say $e = x + d(y)$ for some $x, y \in J$. As $d(e) = 0$ we see that $d(x) = 0$. So $[e] = [x] \in H^0(A)$. But $[e]$ is idempotent and $[x]$ is nilpotent so $[e] = 0$. This implies that $eA$ is acyclic. Let $M \coloneq \bigoplus_{[e] \neq 0} eA$. Then $M$ is quasi-isomorphic to $A$ and so there is a quasi-isomorphism $A' \coloneq \End_A(M) \simeq \End_A(A) = A$. It is clear that the differential kills every idempotent of $A'$ and that $A'$ has no acyclic summands. Hence by the above argument $J(A') = J(A')_+$ and by construction we see that $A'/J \cong A/J_+$.
\end{proof}

\begin{theorem} \label{additiveinv}
Let $A$ be a smooth finite-dimensional DG-algebra such that $A/J_+$ is $k$-separable and $d(e) = 0$ for every idempotent $e$ of $A$. 
If $U$ is any additive invariant, there is an equivalence 
\[
U(A) \simeq U(A/J_+)
\]

\end{theorem}

\begin{proof}
By Lemma \ref{J=J_+}, we can assume that $J = J_+$. By Lemma \ref{nicetrick}, it is enough to show that $K_0(A^{op} \otimes A) \to K_0(A/J^{op} \otimes A/J)$ is an isomorphism. As $A/J$ is $k$-separable, $A^e/J(A^e) = A/J^{op} \otimes A/J$ and every primitive idempotent of $A^e$ is of the form $e \otimes f$ for idempotents $e,f$ of $A$. So the differential on $A^e$ kills each idempotent. Since $A^e$ is smooth, Corollary \ref{smoothK0iso} implies the induced map $K_0(A^e) \to K_0(A^e/J(A^e))$ is an isomorphism.
\end{proof}

Our results apply to DG-quivers with relations. 

\begin{cor} \label{quivercor}
Let $A$ be a DG-quiver with relations. Then there is an isomorphism
\[
 \mathbb{Z}^{N_{proj}} \simeq K_0(A) \xrightarrow{\sim} K_0(A/J_+) \simeq \mathbb{Z}^{N_{sim}}. 
\]
If $A$ is additionally regular, then for every additive invariant $U$
\[
U(A) \simeq U(A/J_+) \simeq U(k)^{N_{sim}}
\]
\end{cor}

\begin{proof}
Suppose $A =kQ/I$ and let $kQ_0$ denote the vertex subalgebra. Then there are maps of DG-algebras
\[
kQ_0 \to A \to A/J_+
\]
whose composite is a split epi. Therefore the map $H^0(A) \to H^0(A/J_+) = A/J_+$ is epi and so Corollary \ref{K0isocor} applies and the first isomorphism holds. For the second, note that $A/J_+ \cong k^{N_{sim}}$ is $k$-separable. Hence $A$ is also smooth and Theorem \ref{additiveinv} applies. 
\end{proof}

\begin{remark}
In general $kQ_0 \neq A/J_+$ since its possible for idempotents to be contained in $J_+$. e.g. Take $k[x]/x^2$ with $\lvert x \rvert = -1$ and differential sending $x \mapsto 1$. However by Lemma \ref{J=J_+}, $A$ is quasi-isomorphic to a DG-quiver with relations $A'$ such that $K_0(A') \simeq K_0(kQ'_0)$.
\end{remark}

\section{Phantoms} \label{phantoms}

A phantom is a smooth and proper DG-algebra such that all of its additive invariants vanish. This is equivalent to the vanishing of its endomorphism ring in the category of noncommutative motives in the sense of \cite{Tab15}. Hence a smooth and proper DG-algebra $A$ is a phantom if and only if $K_0(A^e) = 0$. 

\begin{remark}
It follows from \cite[Theorem 6.2]{Efi23}, that every phantom is quasi-isomorphic to a finite-dimensional DG-algebra. This is a surprising result which disproves \cite[Conjecture 3.7]{Orl20} in the strongest possible way. See \cite{Kra23},  \cite{liu26} for the simplest examples of phantoms.
\end{remark}

\begin{prop}
Let $A$ be a finite-dimensional DG-algebra. Then $A$ is a phantom if and only if $K_0(A) = 0$. 
\end{prop}

\begin{proof}
If $A$ is a smooth finite-dimensional DG-algebra such that $K_0(A) = 0$, then it satisfies condition iii) of \cite[Theorem 6.2]{Efi23} by \cite[Theorem 3.3]{Orl20}. So condition ii) holds and $[A] \in K_0(A^e)$ is in the image of $K_0(A) \otimes K_0(A^{op}) = 0$ and so $[A] = 0$. But $[A]$ corresponds to the identity morphism in the category motives as in \cite{Tab15}. Therefore $U_{\add}(A) = 0$ and $A$ is a phantom. 
\end{proof}

\begin{remark}
Parts 1), and 2) of Theorem \ref{introphantom} follow immediately from Theorem \ref{lowerbound} and Remark \ref{idempotentremark}. 
\end{remark}

\begin{prop} \label{basicphantom}
Suppose $A$ is a finite-dimensional phantom then 
\[
A/J_+ \cong \End_{D_1}(N_1) \times \dots \times \End_{D_n}(N_n)
\]
for finite dimensional division algebras $D_i$ where $N_i \in \mathcal{D}(D_i)$ have Euler characteristic zero. In particular, the underlying algebra of $A$ is not basic.
\end{prop}

\begin{proof}
By \cite[Proposition 2.15]{Orl20}, there are division algebras $D_i$ and $N_i \in \mathcal{D}(D_i)$ such that the above isomorphism holds. There is an equivalence $\mathcal{D}(\End(N_i)) \simeq \mathcal{D}(D_i)$ where $\End(N_i) \mapsto N_i$. Note $N_i \simeq H^\ast(N_i) \cong \bigoplus_{j \in \mathbb{Z}} \Sigma^{j} D_i^{(r_j)}$ for some $r_j \geq 0$. The map $K_0(A) \to K_0(A/J_+) 
\cong K_0(D_1 \times \dots \times D_n)$ sends $[A] \mapsto \sum_{i=1}^n \sum_{j \in \mathbb{Z}} (-1)^{j} r_j [D_i]$. If $K_0(A) = 0$, then the image of $[A]$ must also be zero. So the Euler characteristic of $N_i$,  $\sum_{j \in \mathbb{Z}} (-1)^{j} r_j = 0$ for all $i$. If $A$ were basic, $A/J_+$ would be a product of division algebras which each have non-vanishing Euler characteristic.
\end{proof}

\begin{remark}
Proposition \ref{basicphantom} and the existence of finite-dimensional phantoms implies there is a finite-dimensional DG-algebra which is not derived equivalent to any finite-dimensional DG-algebra whose underlying algebra is basic. 
\end{remark}

\section{Hochschild Homology}

Being a phantom has consequences for the Hochschild homology of the cohomology of a phantom. 

\begin{remark} \label{sseqrem}
Let $A$ be a DG-algebra and view $H^\ast(A)$ as a DG-algebra with zero differential. Then the Hochschild homology of $H^\ast(A)$ admits a bigrading $HH_{n}(A) \simeq \bigoplus_{q \in \mathbb{Z}} HH_n(H^\ast(A))^q$ coming from the decomposition of the Hochschild chain complex. There is a spectral sequence converging to $HH_\ast(A)$ with $E_2^{p,q} = HH_{p-q}(H^\ast(A))^q$.
\end{remark}

\begin{lemma} \label{nonvanish}
Let $A$ be a proper DG-algebra. Then $\dim_k HH_0(H^\ast(A)) \geq N_{proj} \neq 0$.
\end{lemma}

\begin{proof}
Let $H = H^\ast(A)$ viewed as a DG-algebra with zero differential. Then using the Hochschild complex one can see that $HH_0(H)^0 =  H^0/\Sigma [H^i,H^{-i}]$ where the bracket denotes the graded commutator. Let $J$ denote the radical of $H$. Then the map $H \to H/J$ induces a surjection $HH_0(H)^0 \twoheadrightarrow HH_0(H/J)^0$. Since $H/J$ is a product of $N_{proj}$ graded simple algebras it is enough to show that $HH_0(S)^0 \neq 0$ for any finite-dimensional graded simple algebra $S$ viewed as a DG-algebra with zero differential. This follows from a graded version of the usual trace argument for non-vanishing of zero-th Hochschild homology of finite-dimensional algebras. \end{proof}

\begin{prop}
Let $A$ be a proper DG-algebra such that $HH_1(H^\ast(A)) = 0$, then $HH_\ast(A) 
\neq 0$.
\end{prop}

\begin{proof}
Consider the spectral sequence from Remark \ref{sseqrem}. If $HH_1(H^\ast(A)) = 0$, then $E_2^{q,q-1} = HH_1(H^\ast(A))^q = 0$ for all $q$. Hence $E_r^{q,q-1} = 0$ for all $r \geq 2$. But this implies every differential $d_r: E_r^{r,r-1} \to E_r^{0,0}$ vanishes. For degree reasons every differential out of $E_r^{0,0}$ also vanishes. Hence $E_\infty^{0,0} = E_2^{0,0} = HH_0(H^\ast(A))^0 \neq 0$ by Lemma \ref{nonvanish}. 
\end{proof}

\begin{prop} \label{smoothimpliessmooth}
Let $A$ be a proper DG-algebra over a perfect field such that $H^\ast(A)$ has finite global dimension. Then $A$ is smooth. 
\end{prop}

\begin{proof}
If $H^\ast(A)$ has finite global dimension, it also has finite global dimension as a graded algebra. Since $k$ is perfect, it follows that $H^\ast(A)$ admits a finite projective resolution as a graded bimodule over $H^\ast(A^e) = H^\ast(A)^e$. This finite projective resolution lifts to a finite semi-free resolution of $A$ as a DG-bimodule over $A^e$ by e.g. \cite[Theorem on pp. 69 c)]{kel94}. Hence $A \in \mathcal{D}^{\perf}(A^e)$.
\end{proof}

\begin{theorem}
Let $A$ be a proper DG-algebra over a perfect field such that $H^\ast(A)$ has finite global dimension. Then 
\[
HH_\ast(A) \simeq HH_\ast(H^\ast(A)) \simeq HH_0(H^\ast(A))
\]
\end{theorem}

\begin{proof}
Suppose $H^\ast(A)$ has finite global dimension. By Proposition \ref{smoothimpliessmooth}, $H^\ast(A)$ is  smooth viewed as a finite-dimensional DG-algebra with zero differential. By Proposition \ref{niceinv}, $HH_n(H^\ast(A)) = 0$ for $n \neq 0$. Therefore the $E_2$ page of the spectral sequence in Remark \ref{sseqrem} is non-zero except for $E_2^{q,q} = HH_0(H^\ast(A))^q$. Therefore all the differentials vanish and $E_\infty^{p,q} = E_2^{p,q}$. Hence convergence implies that $HH_n(A) = 0$ for $n \neq 0$ and $HH_0(A) \simeq HH_0(H^\ast(A))$. 
\end{proof}

\bibliographystyle{alpha}
\bibliography{biblio.bib}

@article{Goo26,
  author    = {Goodbody, I.},
  title     = {Approximable Triangulated Categories and Reflexive {DG}-Categories},
  journal   = {Applied Categorical Structures},
  volume    = {34},
  number    = {3},
  pages     = {19},
  year      = {2026},
  doi       = {10.1007/s10485-026-09856-0},
  url       = {https://link.springer.com/article/10.1007/s10485-026-09856-0},
  note      = {With an appendix by T. Raedschelders and G. Stevenson}
}

@article{RS20,
      title={Proper connective differential graded algebras and their geometric realizations}, 
      author={T. Raedschelders and G. Stevenson},
      year={2022},
      journal={Eur. J. Math.},
      volume = {8},
      pages = {574–598}
}

@article{KL15,
    author = {A. Kuznetsov and V.A. Lunts},
    title = {Categorical Resolutions of Irrational Singularities},
    journal = {Int. Math. Res. Not.},
    volume = {2015},
    number = {13},
    pages = {4536-4625},
    year = {2015}
}

@article{BDG17,
    author = {Burban, I. and Drozd, Y. and Gavran, V.},
    title = {Singular Curves and Quasi-hereditary Algebras},
    journal = {International Mathematics Research Notices},
    volume = {2017},
    number = {3},
    pages = {895-920},
    year = {2017},
    month = {02},
    issn = {1073-7928},
    doi = {10.1093/imrn/rnw045},
    url = {https://doi.org/10.1093/imrn/rnw045},
    eprint = {https://academic.oup.com/imrn/article-pdf/2017/3/895/10017761/rnw045.pdf},
}

@article{kra23,
  title = {A phantom on a rational surface},
  author = {Krah, J.},
  year = {2024},
  journal = {{Invent. Math.}},
  volume = {235},
  pages = {1009–1018}
}

@article{goo23,
year = {2024},
author = {Goodbody, I.},
title = {Reflecting perfection for finite-dimensional differential graded algebras},
journal = {Bull. Lond. Math. Soc.},
volume = {56},
number ={12},
pages = {3689-3707},
}

@book{Tab15, 
    author={G. Tabuada},
    series = {Univ. Lecture Ser.},
    title={Noncommutative motives}, 
    publisher={Amer. Math. Soc.}, 
    volume = {63},
    year={2015}
}

@article{LS14,
  author    = {V. A. Lunts and O. M. Schn{\"{u}}rer},
  title     = {Smoothness of equivariant derived categories},
  journal   = {Proc. London Math. Soc. (3)},
  volume    = {108},
  number    = {5},
  pages     = {1226--1276},
  year      = {2014},
  doi       = {10.1112/plms/pdt053},
  url       = {https://doi.org/10.1112/plms/pdt053},
  archivePrefix = {arXiv},
  eprint    = {1205.3132},
  primaryClass = {math.AG}
}

@article{VB02,
    AUTHOR = {Bondal, A. and van den Bergh, M.},
     TITLE = {Generators and representability of functors in commutative and
              noncommutative geometry},
   JOURNAL = {Mosc. Math. J.},
  FJOURNAL = {Moscow Mathematical Journal},
    VOLUME = {3},
      YEAR = {2003},
    NUMBER = {1},
     PAGES = {1--36, 258},
      ISSN = {1609-3321,1609-4514},
   MRCLASS = {18E30 (14F05)},
  MRNUMBER = {1996800},
MRREVIEWER = {Ioannis\ Emmanouil},
       DOI = {10.17323/1609-4514-2003-3-1-1-36},
       URL = {https://doi.org/10.17323/1609-4514-2003-3-1-1-36},
}

@incollection{Efi23,
  author    = {Efimov, A.},
  title     = {Smooth compactifications in derived non-commutative geometry},
  booktitle = {European Congress of Mathematics ({ECM})},
  pages     = {535--551},
  publisher = {EMS Press},
  year      = {2023},
  doi       = {10.4171/8ECM/36},
  url       = {https://ems.press/books/standalone/262/5200}
}

@article{Orl20,
	doi = {10.1016/j.aim.2020.107096},
  
	url = {https://doi.org/10.1016\%2Fj.aim.2020.107096},
  
	year = 2020,
	month = {jun},
  
	publisher = {Elsevier {BV}
},
  
	volume = {366},
  
	pages = {107096},
  
	author = {D. Orlov},
  
	title = {Finite-dimensional differential graded algebras and their geometric realizations},
  
	journal = {Advances in Mathematics}
}

@article{liu26,
      title={Phantom categories on 3-vertex directed DG quivers}, 
      author={Y. Liu and Y. Shen},
      year={2026},
      journal = {arXiv preprint: 2609.27191}
}

@article{KS22,
  doi = {10.48550/ARXIV.2211.09418},
  
  url = {https://arxiv.org/abs/2211.09418},
  
  author = {Kuznetsov, A. and Shinder, E.},
  
  
  title = {Homologically finite-dimensional objects in triangulated categories},

  
  year = {2025},
  volume =  31, 
  number = 27,
  journal={Selecta Math.}
}

@article{Kel94, title={Deriving DG categories}, volume={27}, DOI={10.24033/asens.1689}, number={1}, journal={Annales scientifiques de l’École normale supérieure}, author={Keller, B.}, year={1994}, pages={63–102}}

@article{Wei16,
  author    = {Z. Wei},
  title     = {The full exceptional collections of categorical resolutions of curves},
  journal   = {Journal of Pure and Applied Algebra},
  volume    = {220},
  number    = {9},
  pages     = {3332--3344},
  year      = {2016},
  doi       = {10.1016/j.jpaa.2016.02.018},
  url       = {https://www.sciencedirect.com/science/article/pii/S0022404916000505},
  eprint    = {1506.02991},
  archivePrefix = {arXiv},
  primaryClass = {math.AG}
}

@article{Orl14,
  author    = {Orlov, D.},
  title     = {Smooth and proper noncommutative schemes and gluing of {DG} categories},
  journal   = {Advances in Mathematics},
  volume    = {302},
  pages     = {803--849},
  year      = {2016}
}

@article{LZ26,
  title={Orlov's rationality conjecture for surfaces},
  author={Lin, X. and Zhang, S.},
  journal={arXiv preprint: 2609.21807},
  year={2026}
}

@article{NK11,
      title={Weight structures and simple dg modules for positive dg algebras}, 
      author={B. Keller and P. Nicol\'as},
      year={2013},
      journal = {Int. Math. Res. Not.},
      volume = {2013},
      number = {5},
      pages = {1028–1078}
}

\end{document}